\documentclass[11pt]{amsart}
\usepackage{times}
\usepackage{latexsym}
\usepackage{amssymb}
\usepackage{mathtools}
\usepackage{url}
\usepackage[all]{xy}

\newtheorem{thm}{Theorem}[section]

\newtheorem{cor}[thm]{Corollary}

\newtheorem{lem}[thm]{Lemma}
\newtheorem{que}[thm]{Question}

\theoremstyle{definition}
\newtheorem{rem}[thm]{Remark}

\newtheorem{ex}[thm]{Example}

\newcommand{\N}{\mathbb{N}}
\newcommand{\Q}{\mathbb{Q}}
\newcommand{\R}{\mathbb{R}}
\newcommand{\Z}{\mathbb{Z}}

\title[Domination between different products]
{Domination between different products \\ and finiteness of associated semi-norms}

\author{Christoforos Neofytidis}
\address{Department of Mathematics, Ohio State University, Columbus, OH 43210, USA}
\email{neofytidis.1@osu.edu}
\date{\today}
\thanks{I am grateful to M. Gromov for stimulating questions and discussions. 
The hospitality and support of the IH\'ES where part of this project was carried out is also gratefully acknowledged.}

\begin{document}

\begin{abstract}
In this note we determine all possible dominations between different products of manifolds, when none of the factors of the codomain is dominated by products. As a consequence, we determine the finiteness of every product-associated functorial semi-norm on the fundamental classes of the aforementioned products. These results give partial answers to questions of M. Gromov. 
\end{abstract}

\maketitle

\section{Motivation and Results}

A finite functorial semi-norm in degree $k\in\N$ singular homology is a semi-norm
\[
\nu\colon H_k(X;\R)\longrightarrow[0,\infty)
\]
for every topological space $X$, where ``functorial" means that the semi-norm $\nu$ is not increasing under induced homomorphisms $f_*\colon H_*(Y)\longrightarrow H_*(X)$ for all continuous maps $f\colon Y\longrightarrow X$.

In \cite[Chapter 5G$_+$]{gromovmetric} Gromov suggested (originally using the Euler characteristic of products of surfaces, see below) the following construction of product-associated semi-norms on homology classes of a topological space $X$: Let $\nu$ be a finite functorial semi-norm on the fundamental classes of products of closed oriented $k$-manifolds. For a homology class $\alpha\in H_{k\ell}(X;\Z)$, $\ell\in\N$, define
\begin{equation}\label{definition}
\nu_{k,\ell}(\alpha):=\inf_{d,M_1\times\cdots\times M_\ell,f}  \frac{\nu([M_1\times\cdots\times M_\ell])}{d},
\end{equation}
where the infimum is taken over all $d=1,2,...$, all products $M_1\times\cdots\times M_\ell$ of closed oriented $k$-manifolds $M_i$, and all continuous maps $f\colon M_1\times\cdots\times M_\ell\longrightarrow X$ such that $f_*([M_1\times\cdots\times M_\ell])=d\cdot\alpha$. For $\ell=1$ we have a trivial product with only one factor.

The idea of extending $\nu$ from the category of products of $k$-manifolds to $\nu_{k,\ell}$, i.e. to any $k\ell$-dimensional integral homology class and any topological space $X$, stems from the following immediate property that every functorial semi-norm satisfies:

\begin{lem}[Mapping Lemma]\label{property}
Let $k\geq1$ and $\nu$ be a (finite) functorial semi-norm on the fundamental classes of products of closed oriented $k$-manifolds $M_i$. If $f\colon M_1\times\cdots\times M_\ell\longrightarrow M'_1\times\cdots\times M'_\ell$ is a map of degree $d$, then $\nu(M_1\times\cdots\times M_\ell)\geq d\cdot\nu(M'_1\times\cdots\times M'_\ell)$.
\end{lem}

Similarly to Thurston, who used the Euler  characteristic of embedded surfaces in a $3$-manifold $M$ to define a norm in $H_2(M)$, Gromov's original example in \cite{gromovmetric} 
is a norm in degree $2\ell$ homology where $\nu$ is the absolute value of the Euler characteristic $\chi$ of (products of) surfaces. Namely, for a 
space $X$ and $\alpha\in H_{2\ell}(X;\Z)$, $\ell\in\N$, the {\em (product) Euler characteristic norm} is defined as
\begin{equation}
\chi_{2,\ell}(\alpha):=\inf_{d,\Sigma_1\times\cdots\times\Sigma_\ell,f}  \frac{|\chi(\Sigma_1\times\cdots\times\Sigma_\ell)|}{d},
\end{equation}
where the infimum is taken over all $d=1,2,...$, all products $\Sigma_1\times\cdots\times\Sigma_\ell$ of closed hyperbolic surfaces $\Sigma_i$, and all continuous maps $f\colon\Sigma_1\times\cdots\times\Sigma_\ell\longrightarrow X$ such that $f_*([\Sigma_1\times\cdots\times\Sigma_\ell])=d\cdot\alpha$.
Indeed, the Euler characteristic satisfies Lemma \ref{property} for maps between (products of) surfaces; see the Mapping Lemmas in \cite[Sections 5.35--36]{gromovmetric}.

Gromov asked when the product Euler characteristic norm is finite, writing~\cite[page 301]{gromovmetric}
\begin{center}
{\em `` it is unclear which classes in $H_{2\ell}(X)$\\ come from (mapped) products of surfaces"}.
\end{center}
The obvious generalization of Gromov's question is:

\begin{que}\label{q:Gromovnorm}
Let $\nu$ be a finite functorial semi-norm on the fundamental classes of products of closed oriented $k$-manifolds. For which spaces $X$ and which homology classes $\alpha\in H_{k\ell}(X;\Z)$, $\ell\in\N$, is $\nu_{k,\ell}(\alpha)$ finite?
\end{que}

Gromov predicted that the product Euler characteristic norm is infinite on many $2\ell$-dimensional fundamental classes ($\ell>1$), pointing out the fundamental classes of irreducible locally symmetric spaces as potential candidates. That prediction has since been verified by Kotschick and L\"oh, who proved that irreducible locally symmetric spaces of non-compact type do not admit maps of non-zero degree from direct products (whose factors are of any dimension, not necessarily surfaces); cf.~\cite[Corollary 4.2]{KL}.  

The topic of realizing (co-)homology classes by direct products of manifolds is a special case of a classical problem of Steenrod~\cite[Problem 25]{EilenbergSteenrod}. When the target homology class is the fundamental class of a manifold, we deal with maps of non-zero degree. We say that $M$ {\em dominates} $N$, and write $M\geq N$, if there is a continuous map $f\colon M\longrightarrow N$ of non-zero degree, that is $f_*([M])=\deg(f)\cdot[N]$ in homology or equivalently $f^*(\omega_N)=\deg(f)\cdot\omega_M$ in cohomology (as usual, $\omega_M\in H^{\dim M}(M)$ denotes 
the cohomological fundamental class of $M$). 

The following question, posed to me by M. Gromov, is essential in order to understand the finiteness of $\nu_{k,\ell}$ on the fundamental classes of arbitrary products, and has also independent interest on the level of domination between manifolds:

\begin{que}\label{q:Gromov}
Let $X_1\times\cdots\times X_m$ be a Cartesian product of closed oriented manifolds $X_i$ of positive dimensions. Which other non-trivial products dominate $X_1\times\cdots\times X_m$?
\end{que}

In this paper we give a complete answer to Question \ref{q:Gromov} when none of the factors $X_i$ is dominated by products:

\begin{thm}\label{t:main}
 Suppose $X_1,...,X_m,Y_1,...,Y_\ell$ are closed oriented manifolds of positive dimensions, such that $X_1,...,X_m$ are not dominated by non-trivial direct products and $\dim(X_1\times\cdots\times X_m)=\dim(Y_1\times\cdots\times Y_\ell)$. Then 
$ Y_1\times\cdots\times Y_\ell \geq X_1\times\cdots\times X_m$
 if and only if 
 $Y_i\geq X_{a_{i1}}\times\cdots\times X_{a_{i\xi_i}}$ for all $i=1,...,\ell$,
where $\xi_i\geq 1$, $a_{ij}\in\{1,...,m\}$ and $a_{ij}\neq a_{i'j'}$ if $(i,j)\neq(i',j')$.
\end{thm}

In particular, we obtain an answer to Question \ref{q:Gromovnorm} for fundamental classes of products whose factors are not dominated by products:

\begin{cor}\label{c:finiteness}
Let $X_1,...,X_m$ be closed oriented manifolds of positive dimensions that are not dominated by non-trivial direct products and $\dim(X_1\times\cdots\times X_m)=k\ell$, for some $k, \ell\in\N$. The following are equivalent:
\begin{itemize}
\item[(i)] $X_1\times\cdots\times X_m$ is a product with $\ell$ factors of closed oriented $k$-manifolds.
\item[(ii)] Every semi-norm $\nu_{k,\ell}$ is finite on $[X_1\times\cdots\times X_m]$.
\item[(iii)] There is a finite semi-norm $\nu_{k,\ell}$ on $[X_1\times\cdots\times X_m]$.
\end{itemize}
\end{cor}

Note that if $\alpha=[M_1\times\cdots\times M_\ell]$ in (\ref{definition}), then obviously
\[\nu_{k,\ell}(\alpha)=\nu([M_1\times\cdots\times M_\ell]).\]
 Thus, the equivalent conditions (i)-(iii) in Corollary \ref{c:finiteness} are moreover equivalent to 
\[\nu_{k,\ell}([X_1\times\cdots\times X_m])=\nu([X_1\times\cdots\times X_m]),\]
for every finite semi-norm $\nu$.

\section{Proofs}

We now prove Theorem \ref{t:main} and Corollary \ref{c:finiteness}. The proof of Theorem \ref{t:main} uses Thom's work~\cite{Thom} on the Steenrod problem about realizing homology classes by closed manifolds~\cite[Problem 25]{EilenbergSteenrod}. Thom's celebrated realization theorem states that, given a topological space $X$ and a homology class $\alpha\in H_k(X;\Z)$, there is a closed oriented smooth $k$-dimensional manifold $M$ and a continuous map $f\colon M\longrightarrow X$ such that $f_*([M])=d\cdot \alpha$, for some non-zero integer $d$. Or, equivalently, if one starts with a cohomology class $\beta\in H^k(X;\Z)$, then $f^*(\beta)=d\cdot \omega_M$.

\begin{proof}[Proof of Theorem \ref{t:main}]
The `` if " direction is trivial and so we prove the converse. 

Let $f\colon Y_1\times\cdots\times Y_\ell\longrightarrow X_1\times\cdots\times X_m$ be a map of non-zero degree, and denote by $p_{X_i}\colon X_1\times\cdots\times X_m\longrightarrow X_i$ the projection to the $i$-th factor. Then $f^*(p_{X_i}^*(\omega_{X_i}))$ is not trivial 
and, since the $X_i$ are not dominated by products, Thom's theorem~\cite{Thom} implies that $f^*(p_{X_i}^*(\omega_{X_i}))$ belongs in
\[
H^{\dim X_i}(Y_1;\Q)\oplus\cdots\oplus H^{\dim X_i}(Y_j;\Q)\oplus\cdots\oplus H^{\dim X_i}(Y_\ell;\Q).\]
 Indeed, suppose $p_{X_i}^*(\omega_{X_i})$ maps non-trivially under $f^*$ in some 
\[
H^{k_1}(Y_1;\Q)\otimes\cdots\otimes H^{k_j}(Y_j;\Q)\otimes\cdots\otimes H^{k_\ell}(Y_\ell;\Q),
\]
where $0\leq k_1,...,k_\ell <\dim X_i$ and $k_1+\cdots+k_\ell=\dim X_i$. Then, by Thom's theorem, there exist two closed oriented manifolds $W_1$ and $W_2$ of positive dimensions, with $\dim W_1+\dim W_2=\dim X_i$, and a continuous map $g\colon W_1\times W_2\longrightarrow X_i$ such that $g^*(f^*(p_{X_i}^*(\omega_{X_i})))=d\cdot\omega_{W_1\times W_2}\in H^{\dim X_i}(W_1\times W_2)$, for some non-zero integer $d$. Thus $W_1\times W_2\geq X_i$, which contradicts our assumption that $X_i$ is not dominated by products.

Thus, we have
\begin{equation}\label{equation}
f^*(p_{X_i}^*(\omega_{X_i}))=\sum_{j=1}^\ell (1\times\cdots\times 1\times\alpha_j^{X_i}\times 1\times\cdots\times 1),
\end{equation}
where $\alpha_j^{X_i}\in H^{\dim X_i}(Y_j;\Q)$. 
\subsection{A Reduction: $\ell\leq m$}\label{reduction} 

We first observe that (\ref{equation}) implies that $\ell$ can be at most $m$, otherwise the number of factors in the codomain $X_1\times\cdots\times X_m$ of $f$ would not suffice to give
\begin{eqnarray}\label{equation1}
\begin{split}
f^*(p_{X_1}^*(\omega_{X_1}))\cup\cdots\cup f^*(p_{X_m}^*(\omega_{X_m})) 
& =   \deg(f)\cdot\omega_{Y_1\times\cdots\times Y_\ell} \\
& =  \deg(f)\cdot\omega_{Y_1}\times\cdots\times\omega_{Y_\ell}
\end{split}
\end{eqnarray}

Thus we split the proof into the following cases:

\subsection{Case I: $\ell=m$}\label{caseI}

In this case, (\ref{equation}) and (\ref{equation1}) imply that for each $X_i$ there exists at least one $Y_j$ such that 
$$\dim X_i =\dim Y_j \ \text{ and} \ f^*(p_{X_i}^*(\omega_{X_i}))\neq 0\in H^{\dim Y_j}(Y_j;\Q).$$ 
This means that $Y_j\geq X_i$ through the composite map
\[
Y_j\xhookrightarrow{\iota_{Y_j}} Y_1\times\cdots\times Y_m\stackrel{f}\longrightarrow X_1\times\cdots\times X_m\xrightarrow{p_{X_i}} X_i,
\]
where $\iota_{Y_j}\colon Y_j\hookrightarrow Y_1\times\cdots\times Y_m$ is the inclusion.

The assumption that $\ell=m$ and (\ref{equation1}) imply moreover that for each $i\neq i'$ there exist $j\neq j'$ with $Y_j\geq X_i$ and $Y_{j'}\geq X_{i'}$. Thus, after reordering the $Y_i$ if necessary, we conclude that $Y_i\geq X_i$ for all $i=1,2,...,m$.

\subsection{Case II: $\ell<m$}\label{caseII}

In this case, (\ref{equation}) and (\ref{equation1}) imply that for some $Y_i$ there exist $X_{a_{i1}},...,$ $X_{a_{i\xi_i}}$, $\xi_i\geq 2$,  among the $X_1,...,X_m$, such that 
$$\sum_{j=1}^{\xi_i}\dim X_{a_{ij}}=\dim Y_i \ $$
and 
$$f^*(p_{X_{a_{i1}}\times\cdots\times X_{a_{i\xi_i}}}^*(\omega_{X_{a_{i1}}\times\cdots\times X_{a_{i\xi_i}}}))\neq 0 \in H^{\dim Y_i}(Y_i;\Q),$$
where 
$p_{X_{a_{i1}}\times\cdots\times X_{a_{i\xi_i}}}\colon X_1\times\cdots\times X_m \longrightarrow X_{a_{i1}}\times\cdots\times X_{a_{i\xi_i}}$ is the projection. This means that $Y_i\geq X_{a_{i1}}\times\cdots\times X_{a_{i\xi_i}}$ through the composite map
\[
Y_i\xhookrightarrow{\iota_{Y_i}} Y_1\times\cdots\times Y_\ell\stackrel{f}\longrightarrow X_1\times\cdots\times X_m \xrightarrow{p_{X_{a_{i1}}\times\cdots\times X_{a_{i\xi_i}}}} X_{a_{i1}}\times\cdots\times X_{a_{i\xi_i}},
\]
where $\iota_{Y_i}\colon Y_j\hookrightarrow Y_1\times\cdots\times Y_\ell$ is the inclusion.

Now, by the naturality of the cup product, we obtain
$$
Y_1\times\cdots\times Y_{i-1}\times Y_{i+1}\times\cdots\times Y_\ell \geq \prod_{\substack{q=1 \\ q\notin\{a_{i1},...,a_{i\xi_i}\}}}^m X_q,
$$
and so Reduction \ref{reduction} implies that $\ell-1\leq m-\xi_i$. If $\ell-1=m-\xi_i$, then the result follows by Case I. If $\ell-1< m-\xi_i$, then we repeat the argument as in Case II, to find some $i'\neq i$ and some $\xi_{i'}\geq 2$ such that $Y_{i'}\geq X_{a_{i'1}}\times\cdots\times X_{a_{i'\xi_{i'}}}$ (where $a_{ij}\neq a_{i'j'}$ for all $1\leq j\leq \xi_i,1\leq j' \leq \xi_{i'}$). We then have $\ell-2\leq m-\xi_i-\xi_{i'}$ and we finish the proof by iterating the process.
\end{proof}

\begin{proof}[Proof of Corollary \ref{c:finiteness}] 
(i) $\Rightarrow$ (ii)
If $X_1\times\cdots\times X_m$ can be written as a product with $\ell$ factors of closed oriented $k$-manifolds $Y_i$, then clearly every semi-norm $\nu_{k,\ell}$ is finite on $[X_1\times\cdots\times X_m]$, because 
\[
\nu_{k,\ell}([X_1\times\cdots\times X_m])=\nu_{k,\ell}([Y_1\times\cdots\times Y_\ell])=\nu([Y_1\times\cdots\times Y_\ell])
\]
and $\nu([Y_1\times\cdots\times Y_\ell])$ is finite by assumption. 

(ii) $\Rightarrow$ (iii) This implication holds trivially.

(iii) $\Rightarrow$ (i) Suppose that some semi-norm $\nu_{k,\ell}([X_1\times\cdots\times X_m])$ is finite. This means that there exist closed oriented $k$-manifolds $Y_1,...,Y_\ell$ such that 
\[
Y_1\times\cdots\times Y_\ell\geq X_1\times\cdots\times X_m.
\]
Then Theorem \ref{t:main} implies that each $Y_i$ dominates a different (and possibly containing only one factor) subproduct $X_{a_{i1}}\times\cdots\times X_{a_{i\xi_i}}\subset X_1\times\cdots\times X_m$. In particular, each $X_{a_{i1}}\times\cdots\times X_{a_{i\xi_i}}$ is a $k$-manifold, and so $X_1\times\cdots\times X_m$ can be written as a product with factors those $\ell$ $k$-manifolds:
\[
(X_{a_{11}}\times\cdots\times X_{a_{1\xi_1}})\times(X_{a_{21}}\times\cdots\times X_{a_{2\xi_2}})\times\cdots\times(X_{a_{\ell1}}\times\cdots\times X_{a_{\ell\xi_\ell}}).
\]
\end{proof}

\begin{rem}
The statements and proofs in this paper are on the level of products of fundamental classes of manifolds. One can naturally generalize Theorem \ref{t:main} to the level of realizing arbitrary products of co-homology classes by other products of co-homology classes and, subsequently, obtain (non-)finiteness results of semi-norms on products of more general co-homology classes instead of fundamental classes of products of manifolds.
\end{rem}

\section{Two illustrative examples}
The key property in this paper is that none of the factors of the codomain is dominated by direct products.
There is a variety of examples of manifolds that are not dominated by products, and techniques to identify such manifolds were developed in the recent years~\cite{KL,KL2,KN,NeoIIPP}. Some large classes of examples are non-positively curved manifolds that are not decomposable as products and certain circle bundles, including low-dimensional aspherical manifolds that possess certain Thurston geometries. So, any combination of those manifolds can be used to construct direct products that fulfill Theorem \ref{t:main} and Corollary \ref{c:finiteness}. 

\begin{ex}\label{ex.1}
Suppose $X_1, X_2, X_3$ are closed oriented manifolds of dimensions 
$\dim(X_1)=3$,  $\dim(X_2)=6$ and $\dim(X_3)=9$.
The possible ordered pairs $(k,\ell)$ such that $k\ell=\dim (X_1 \times X_2 \times X_3) =18$ are 
$$(1,18), \ (2,9), \ (3,6), \ (6,3), \ (9,2), \ (18,1).$$  If the $X_i$ are not dominated by products, then Corollary \ref{c:finiteness} applies: First, since there are three factors, then for any finite functorial semi-norm $\nu$ we obtain
$$\nu_{1,18}([X_1\times X_2\times X_3])=\nu_{2,9}([X_1\times X_2\times X_3])=\nu_{3,6}([X_1\times X_2\times X_2])=\infty.$$
Also $X_1\times X_2\times X_3$ is not a product of three $6$-manifolds, thus we have 
\[
\nu_{6,3}([X_1\times X_2\times X_3])=\infty.
\]
 However, $X_1\times X_2\times X_3$ is a product of the two $9$-dimensional manifolds $X_1\times X_2$ and $X_3$ and so
\[
\nu_{9,2}([X_1\times X_2\times X_3])=\nu([X_1\times X_2\times X_3])<\infty.
\]
Finally, it is immediate by the definition in (\ref{definition}) that 
\[
\nu_{18,1}([X_1\times X_2\times X_3])=\nu([X_1\times X_2\times X_3])<\infty.
\]
\end{ex}

\begin{ex}\label{ex.2}
Let $X_1,...,X_m$, $\dim(X_i)\geq2$, be closed oriented manifolds that are not dominated by products. Suppose
\[
\Sigma_1\times\cdots\times \Sigma_\ell \geq X_1\times\cdots\times X_m,
\]
where $\Sigma_1,...,\Sigma_\ell$ are closed oriented (hyperbolic) surfaces. By Theorem \ref{t:main} (or by~\cite[Theorem 2.3]{KLN}) we conclude that each $X_i$ is a surface, and since the $X_i$ are not dominated by products, we deduce that each $X_i$ is a hyperbolic surface (and also $m=\ell$).

Thus Corollary \ref{c:finiteness} implies that, if $X_1,...,X_m$ are closed oriented manifolds that are not dominated by products, then
\[
\chi_{2,\ell}([X_1\times\cdots\times X_m])=
\begin{cases}
|\chi(X_1\times\cdots\times X_m)| &\text{if each $X_i$ is a surface},\\
\infty & \text{otherwise}.
\end{cases}
\]
This answers Question \ref{q:Gromovnorm} for the product Euler characteristic norm on fundamental classes of products whose factors are not themselves dominated by products.
\end{ex}

\end{document}